\documentclass[a4paper]{amsart}
\usepackage{mymacros}
\usepackage{hyperref}
\usepackage{todonotes}

\hypersetup{final}

\ifdraft{\usepackage[left]{showlabels}}{}

\makeatletter
\newtheorem*{rep@theorem}{\rep@title}
\newcommand{\newreptheorem}[2]{
  \newenvironment{rep#1}[1]{
    \def\rep@title{#2 \ref{##1}}
    \begin{rep@theorem}}
    {\end{rep@theorem}}}
\makeatother

\newreptheorem{ex}{Example}
\newreptheorem{prop}{Proposition}
\newreptheorem{thm}{Theorem}

\begin{document}
\title{Holonomy of Manifolds with Density}
\author{Dmytro Yeroshkin}\thanks{The author received funding from Excellence of Science grant number 30950721, "Symplectic Techniques"}
\email{Dmytro.Yeroshkin@ulb.ac.be}
\address{Geometri\'e Differentielle, Universit\'e Libre de Bruxelles}
\begin{abstract}
In this paper we discuss some examples and general properties of holonomy groups of $\nabla^\phi$ introduced by Wylie and the author, the connection corresponding to the $N=1$ Bakry-\'Emery Ricci curvature, and also Wylie's $\bar{\sec}_f$. In particular we classify all possible holonomy groups in dimension 2 and also provide two infinite families: $SL_n(\bbR)$ and $SO^+(p,q)$.
\end{abstract}
\maketitle

\makeatletter
\providecommand\@dotsep{5}
\makeatother


\section{Introduction}\label{sec:Intro}

In the 1970's, Lichnerowicz \cite{Lichnerowicz70,Lichnerowicz72} studied the Ricci curvature of Riemannian manifolds with a smooth positive density function (usualy denoted $e^{-f}$). This was later generalized in the work of Bakry and \'Emery \cite{BE}. The main object of this study is what is now called Bakry-\'Emery Ricci curvature:
\[
  \Ric_f^N = \Ric + \Hess f - \frac{df\otimes df}{N-n},
\]
where $N$ has traditionally been viewed as a parameter in $(n,\infty]$. More recently several authors have considered $N<n$. See for example \cite{KolesnikovMilman, Milman1, Ohta, WSecDens}.

In \cite{WYDens}, Wylie and the author introduced a torsion-free affine connection whose Ricci tensor is $\Ric_f^1$:
\[
  \nabla^\phi_X Y = \nabla_X Y - d\phi(X)Y - d\phi(Y)X,
\]
where $\phi = \frac{f}{n-1}$, and $\nabla$ denotes the Levi-Civita connection. In that paper, and later in \cite{KWYDens}, the connection was used for results in multiple aspects of geometry. In this paper, our focus will be on the holonomy of this connection.

Given a manifold $M$, with an affine connection $\nabla$, recall that the holonomy group at $p\in M$ is the group of all linear maps $P_\sigma:T_pM\to T_p M$ obtained by parallel translation along a piece-wise $C^1$ loop $\sigma:[0,T]\to M$ based at $p$. That is, solving the equation $\nabla_{\dot{\sigma}(t)} U(t) = 0$, and mapping $U(0)$ to $U(T)$.

In the Rimannian case, the possible holonomy groups of simply-connected manifolds were fully classified by Berger in \cite{BHol}. In the more general setting of torsion-free affine connections, the possible irreducible groups were fully classified by the work of Merkulov and Schwachh\"offer \cite{MeSch} and Bryant \cite{Bryant}.

In \cite{WYDens}, the author and Wylie proved that the holonomy group of $\nabla^\phi$, which we will denote as $Hol^\phi$ lies inside $SL_n(\bbR)$ when $M$ is orientable, and inside $\{A\in GL_n(\bbR)|\det A = \pm 1\}$ when $M$ is non-orientable. The key to that proof was the construction of a (local) volume form that was $\nabla^\phi$-parallel. One observation arising from \cite{WYDens} is that unlike in the Riemannian case, $Hol^\phi$ reducible does not imply decomposable. In particular, in this paper we will see an example where the holonomy group is the Heisenberg group.

The following are the main classification results we obtain:

\begin{thm}\label{thm:2d}
  Every connected subgroup of $SL_2(\bbR)$ can arise as $Hol^\phi$ of a 2-dimensional manifold with density.
\end{thm}

\begin{thm}\label{thm:3d}
  Every 3-dimensional candidate given by Merkulov and Schwachh\"offer \cite{MeSch} that lies inside $SL_3(\bbR)$ can arise as $Hol^\phi$.
\end{thm}

\begin{thm}\label{thm:infam}
    $SL_n(\bbR)$ and $SO^+(p,q)$ can occur as $Hol^\phi$ of a Riemannian manifold with density for every $n$ and every pair $(p,q)$.
\end{thm}

This paper is organized as follows. In Section~\ref{sec:Duality}, we introduce a notion of a dual connection, following the work of Nagaoka and Amari \cite{NADual}, we then show how that notion relates to manifolds with density, and was in fact utilized without being named in \cite{WYDens, KWYDens}. In Section~\ref{sec:General}, we provide some general tools for constructing new examples of $Hol^\phi$, as well as some obstructions that can help in search for specific examples. Finally, in Section~\ref{sec:Ex}, we construct specific examples of weighted holonomy. In the appendix we provide the details of the holonomy computations for all the examples.


\section{Duality}\label{sec:Duality}

The notion of duality of connections was first introduced by Nagaoka and Amari in \cite{NADual}, in the context of smooth families of probability distributions. This construction now plays an important role in the field of Information Geometry. For an introduction to the subject, we refer the reader to \cite{InfGeom}.

\begin{dfn}\label{dfn:Duality}
  Let $(M,g)$ be a Riemannian manifold, two affine connections $\nabla,\nabla^*$ on $M$ are called dual with respect to $g$ if
  \[
    D_X g(Y,Z) = g(\nabla_X Y, Z) + g(Y,\nabla^*_X Z)
  \]
\end{dfn}

We now prove some properties of duality, basing our approach on the work of Lauritzen~\cite{Lauritzen}.

\begin{prop}[\cite{Lauritzen}]\label{prop:DualCod}
  Let $(M,g)$ be a Riemannian manifold, and $\nabla$ a torsion-free affine connection on $M$. Furthermore, suppose that $g$ is $\nabla$-Codazzi, that is $D(X,Y,Z) = (\nabla_X g)(Y,Z)$ is a symmetric 3-tensor. Let $\nabla^*$ be the dual of $\nabla$ with respect to $g$. Then, we obtain the following:

  \begin{enumerate}
  \item $\nabla^*$ is torsion-free

  \item $g$ is $\nabla^*$-Codazzi

  \item $g(\nabla^*_X Y,Z) - g(\nabla_X Y, Z) = D(X,Y,Z)$
  \end{enumerate}
\end{prop}

\begin{rmk}
  The triple $(M,g,\nabla)$ in the proposition above is referred to as a statistical manifold in literature, following the work of Lauritzen \cite{Lauritzen}. In that context, the 3-tensor $D(X,Y,Z)$ is called the Amari-Chentsov tensor.
\end{rmk}

\begin{proof}
  Assume $(M,g,\nabla)$ is a triple as above, and $D = \nabla g$.

  We prove part 3 first:
  \begin{align*}
    g(\nabla^*_X Y,Z) - g(\nabla_X Y,Z) &= D_Xg(Y,Z) - g(Y,\nabla_X Z) - g(\nabla_X Y,Z)\\
                                        &=(\nabla_X g)(Y,Z) = D(X,Y,Z)
  \end{align*}
  For part 2, observe that
  \begin{align*}
    D(X,Y,Z) &= g(\nabla^*_X Y,Z) - g(\nabla_X Y,Z)\\
             &= g(\nabla^*_X Y,Z) - D_Xg(Y,Z) + g(Y,\nabla^*_X Z)\\
             &= -(\nabla^*_X g)(Y,Z)
  \end{align*}
  so, $g$ is $\nabla^*$-Codazzi, with $\nabla^* g = -D$.

  For part 1, we note that
  \begin{align*}
    g(\nabla^*_X Y - \nabla^*_Y X,Z) &= g(\nabla^*_X Y,Z) - g(\nabla^*_Y X,Z)\\
                                     &= D(X,Y,Z) + g(\nabla_X Y,Z) - D(Y,X,Z) - g(\nabla_Y X, Z)\\
                                     &=g(\nabla_X Y - \nabla_Y X, Z)
  \end{align*}
  so, the torsions of $\nabla$ and $\nabla^*$ are the same, and we assumed that $\nabla$ was torsion-free, then so is $\nabla^*$.
\end{proof}

\begin{prop}\label{prop:DualPar}
  Let $(M,g)$ be a Riemannian manifold, and $\nabla,\nabla^*$ affine connections on $M$ dual with respect to $g$. Furthermore, let $\gamma$ be any path along $M$, then
  \[
    g\left(P_\gamma^\nabla X, P_\gamma^{\nabla^*} Y\right) = g(X,Y)
  \]
  where $P_\gamma^\nabla$ (resp. $P_\gamma^{\nabla^*}$) denotes parallel translation along $\gamma$ with respect to $\nabla$ (resp. $\nabla^*$).
\end{prop}

\begin{proof}
  Let $X(t)$ (resp. $Y(t)$) be $\nabla$ (resp. $\nabla^*$) parallel vector fields along $\gamma$, then
  \begin{align*}
    \frac{d}{dt} g(X(t),Y(t)) &= D_{\dot\gamma}g(X(t),Y(t))\\
                              &= g(\nabla_{\dot\gamma}X(t),Y(t)) + g(X(t),\nabla^*_{\dot\gamma}Y(t))\\
                              &= 0
  \end{align*}
  So, $g(X(t),Y(t))$ is constant, which completes the proof.
\end{proof}

For this paper, the following corollary of the above proposition will allow us to use some shortcuts in constructing examples of $Hol^\phi$ in Section~\ref{sec:Ex}.

\begin{cor}\label{cor:DualHol}
  Let $\nabla$, $\nabla^*$ be dual with respect to $g$, then $Hol_p^\nabla\cong Hol_p^{\nabla^*}$, with the isomorphism $\phi:Hol_p^\nabla\to Hol_p^{\nabla^*}$ provided by:
  \[
    \psi(A) = \left(A^*\right)^{-1},
  \]
  where $*$ denotes the adjoint with respect to $g$.
\end{cor}

\begin{prop}\label{prop:Duality}
  The weighted connections for $(M,g,\phi)$ and $(M,e^{-2\phi}g,-\phi)$ (denoted by $\nabla^{\phi}$ and $\nabla^{-\phi}$ respectively) are dual with respect to the metric $e^{-\phi}g$.
\end{prop}

\begin{proof}
  Let $\nabla$ (resp. $\tilde{\nabla}$) be the Levi-Civita connections for $g$ (resp. $e^{-2\phi}g$). Recall that
  \[
    \tilde{\nabla}_X Y = \nabla_X Y - d\phi(X)Y - d\phi(Y)X + g(X,Y)\nabla\phi
  \]
  Therefore,
  \[
    \nabla^{-\phi}_X Y = \nabla_X Y + g(X,Y)\nabla\phi
  \]
  This gives us
  \begin{align*}
    &e^{-\phi}g(\nabla^\phi_X Y, Z) + e^{-\phi}g(Y,\nabla^{-\phi}_X Z)\\
    &\qquad= e^{-\phi}\big[g(\nabla_X Y, Z) - d\phi(X) g(Y,Z) - d\phi(Y)g(X,Z)\\
    &\qquad\qquad\qquad+ g(Y,\nabla_X Z) + g(X,Z)g(Y,\nabla\phi)\big]\\
    &\qquad=e^{-\phi}\left[g(\nabla_X Y,Z) + g(Y,\nabla_X Z) - d\phi(X)g(Y,Z)\right]\\
    &\qquad=D_X\left[e^{-\phi}g(Y,Z)\right]
  \end{align*}
\end{proof}

\begin{rmk}
  From \cite[Proposition~5.30]{WYDens}, we know that $e^{-\phi}g$ is $\nabla^\phi$-Codazzi, and furthermore, the corresponding Amari-Chentsov tensor is
  \[
    D(X,Y,Z) = d\phi(X)e^{-\phi}g(Y,Z) + d\phi(Y)e^{-\phi}g(Z,X) + d\phi(Z) e^{-\phi}g(X,Y).
  \]
\end{rmk}

The following proposition is a more general version of a result that appeared in the author's work with Wylie \cite[Proposition~5.18]{WYDens}, where the analogous result was constructed specifically for the manifolds with density case.

\begin{prop}\label{prop:DualVF}
  Let $(M,g)$ be a Riemannian manifold, and $\nabla,\nabla^*$ be affine connections on $M$ dual with respect to $g$. A vector field $V$ on $M$ is $\nabla$-parallel iff $\alpha = g(V,\cdot)$ is a $\nabla^*$-parallel 1-form.
\end{prop}

\begin{proof}
  \[
    (\nabla^*_X\alpha)(Y) = D_X\alpha(Y) - \alpha(\nabla^*_X Y) = D_Xg(V,Y) - g(V,\nabla^*_X Y) = g(\nabla_X V, Y)
  \]
\end{proof}

We finish the discussion of duality with a simplification of one of Lemma~5.19 from \cite{WYDens}:

\begin{lem}[\cite{WYDens}]\label{lem:WY519}
  If $\nu$ is a $\nabla^\phi$-parallel distribution, then both $\nu$ and $\nu^\perp$ are integrable, where $\nu^\perp$ is the $g$-orthogonal complement of $\nu$.
\end{lem}

\begin{proof}
  Being $g$-orthogonal, means that $\nu$ and $\nu^\perp$ are $e^{-\phi}g$-orthogonal. Then, by Corollary~\ref{cor:DualHol} we know that $\nu$ being $\nabla^\phi$-parallel is equivalent to $\nu^\perp$ being $\nabla^{-\phi}$-parallel. However, both $\nabla^\phi$ and $\nabla^{-\phi}$ are torsion-free, therefore, $\nu$ and $\nu^\perp$ are integrable.
\end{proof}

\begin{rmk}
  Note that the proof in \cite{WYDens} used duality without naming it in equation (5.11).
\end{rmk}


\section{General Observations}\label{sec:General}

The goal of this section is to provide tools for constructing examples, as well as obstructions that help guide the search for new examples. These tools, especially from Subsections~\ref{sec:Opq} and \ref{sec:Struct} will play a key role in constructing examples in Section~\ref{sec:Ex}.

\subsection{Obstructions}\label{sec:Obstruct}

We begin by considering some obstructions to having certain common structures be $\nabla^\phi$-parallel. In particular, we are interested in imposing conditions on having an almost complex structure $J$ or a symplectic structure $\omega$. The existence of such $\nabla^\phi$-parallel tensors would imply $Hol^\phi$ in $SL_n(\bbC)\cdot U(1)$ and $Sp(2n,\bbR)$, respectively.

We begin with a theorem that relates conditions on a compatible triple $(g,J,\omega)$.

\begin{thm}[\cite{ZFInfo}]\label{thm:Compat}  
  Let $\nabla$ be a torsion-free connection on $M$. Let $(g,J,\omega)$ be a compatible triple of metric, almost-complex structure and symplectic form. Then, any two of the following statements imply the third:

  \begin{enumerate}
  \item $g$ is Codazzi with respect to $\nabla$

  \item $J$ is Codazzi with respect to $\nabla$

  \item $\nabla\omega = 0$
  \end{enumerate}
\end{thm}

\begin{prop}\label{prop:Jcompat}
  Let $(M,g,\phi)$ be a manifold with density, and $J$ a $\nabla^\phi$-parallel almost complex structure. Then, either $g$ and $J$ are not compatible, or $\phi$ is constant.
\end{prop}

\begin{proof}
  Suppose that $g$ and $J$ are compatible, that is $g(JX, JY) = g(X,Y)$, then $e^{-\phi} g$ and $J$ are also compatible. From \cite[Proposition~5.30]{WYDens}, we know that $e^{-\phi} g$ is Codazzi with respect to $\nabla^\phi$. Let $\omega(X,Y) = e^{-\phi}g(JX,Y)$, then $(e^{-\phi}g,J,\omega)$ is a compatible triple, where $e^{-\phi}g$ and $J$ are both $\nabla^\phi$-Codazzi. Therefore, by Theorem~\ref{thm:Compat}, $\nabla^\phi\omega = 0$. However, $\left(e^{-\phi}g\right)(X,Y) = \omega(-JX,Y)$, so $\nabla^{\phi}e^{-\phi}g = 0$. However, explicit computation (see \cite{WYDens}) shows that
  \[
    \left(\nabla^\phi_X e^{-\phi}g\right)(Y,Z) = d\phi(X)e^{-\phi}g(Y,Z) + d\phi(Y)e^{-\phi}g(Z,X) + d\phi(Z)e^{-\phi}g(X,Y).
  \]
  If we let $Y=Z$ be orthogonal to $X$ with $|Y|=1$, then we get $e^{-\phi}d\phi(X) = 0$, and so $\phi$ is constant as claimed.
\end{proof}

The symplectic compatibility result is weaker, since $\omega$ being compatible with $e^{-\phi}g$ does not make it compatible with $g$, as is the case for $J$.

\begin{prop}\label{prop:symp}
  Let $(M^{2n},g,\phi)$ be an even dimensional manifold with density, with $n\geq 2$, and $\omega$ be a $\nabla^\phi$-parallel symplectic form. Then, either $\omega$ is not compatible with $e^{-\phi}g$, or $\phi$ is constant.
\end{prop}

\begin{proof}
  Suppose that $\omega$ is both $\nabla^\phi$-parallel, and compatible with $h=e^{-\phi}g$. Then, $\omega(X,Y) = h(JX,Y)$ for some almost-complex structure $J$. Furthermore, since $\nabla^\phi\omega = 0$, and $h$ is $\nabla^\phi$-Codazzi, Theorem~\ref{thm:Compat} tells us that $J$ is also $\nabla^\phi$-Codazzi, that is $\left(\nabla^\phi_X J\right)(Y) = \left(\nabla^\phi_Y J\right)(X)$.

  As in Proposition~\ref{prop:Duality}, we will use $\nabla^{-\phi}$ to denote the weighted connection on $(M,e^{-2\phi}g,-\phi)$, which recall is dual to $\nabla^\phi$ with respect to $h$.

  \begin{align*}
    0 &= \left(\nabla^\phi_X\omega\right)(Y,Z)\\
      &= D_X\omega(Y,Z) - \omega\left(\nabla^\phi_X Y, Z\right) - \omega\left(Y,\nabla^\phi_X,Z\right)\\
      &= D_X h(JY,Z) - h\left(J(\nabla^\phi_X Y),Z\right) - h\left(JY,\nabla^\phi_X Z\right)\\
      &= h\left(\nabla^{-\phi}_X(JY),Z\right) - h\left(J(\nabla^\phi_X Y),Z\right)\\
      &= h\left(\nabla^\phi_X(JY),Z\right) + D(X,JY,Z) - h\left(J(\nabla^\phi_X Y),Z\right)\\
      &= h\left((\nabla^\phi_X J)(Y),Z\right) + D(X,JY,Z) = 0
  \end{align*}

  Switching $X$ and $Y$, we get
  \[
    h\left((\nabla^\phi_Y J)(X),Z\right) + D(Y,JX,Z) = 0
  \]
  Now, using the fact that $J$ is $\nabla^\phi$-Codazzi, we conclude that we must have $D(X,JY,Z) = D(Y,JX,Z)$ for every $X,Y,Z$. From \cite[Proposition~5.30]{WYDens}, we know that $D(X,Y,Z) = d\phi(X)h(Y,Z) + d\phi(Y)h(Z,X) + d\phi(Z)h(X,Y)$, so writing out, we get:
  \begin{align*}
    &d\phi(X)h(JY,Z) + d\phi(JY)h(Z,X) + d\phi(Z)h(X,JY)\\
    =&d\phi(Y)h(JX, Z) + d\phi(JX)h(Z,Y) + d\phi(Z)h(Y,JX)
  \end{align*}
  Now, let $JY=X=Z = \nabla^h\phi$ (the $h$ gradient of $\phi$), then, $d\phi(Y)=d\phi(JX) = 0$ and $h(Y,JX) = h(-JY,X) = -|\nabla^h\phi|^2$, so the above expression becomes:
  \[
    3|\nabla^h\phi|^4 = -|\nabla^h\phi|^4
  \]
  so, $|\nabla^h\phi|=0$, and $\phi$ must be constant as claimed.
\end{proof}

\subsection{$O(p,q)$ Holonomy}\label{sec:Opq}

The main tools we use for constructing examples of $Hol^\phi\subset O(p,q)$ are the work in dimension 2 by Dini~\cite{Dini} and itss generalization to higher dimension by Levi-Civita~\cite{LC-Trasforma} on projectively equivalent Riemannian metrics. Recall that two connections are projectively equivalent if their geodesics are the same up to reparametrization, for us this is relevant since $\nabla^\phi$ and the Levi-Civita connection are projectively equivalent. We present here Dini's result, and a special case of Levi-Civita's result. While the result stated here is weaker than Levi-Civita's full theorem, it suffices for our purposes.

\begin{prop}[\cite{Dini}]\label{prop:Dini}
  Two Riemannian metrics $g,\tilde{g}$ on $M^2$ are projectively equivalent, iff around each point there exists a coordinate chart with coordinates $(u,v)$, such that
  \begin{align*}
    g &= \left(U(u)-V(u)\right)\left(U_1(u)du^2+V_1(v)dv^2\right)\\
    \tilde{g} &= \left(\frac{1}{V(v) - U(u)}\right)\left(\frac{U_1(u)}{U(u)}du^2 + \frac{V_1(v)}{V(v)}dv^2\right)
  \end{align*}
\end{prop}

\begin{prop}[\cite{LC-Trasforma}]\label{prop:LC}
  Given a coordinate chart $(x_1,\ldots,x_n)$, consider two families of functions $\phi_i(x_i)$ and $A_i(x_i)$ satisfying $\phi_1<\phi_2<\cdots<\phi_n$. Let
  \[
    \Pi_i = (\phi_i-\phi_1)\cdots(\phi_i-\phi_{i-1})(\phi_{i+1}-\phi_i)\cdots(\phi_n-\phi_i),
  \]
  and
  \[
    \rho_i = \frac{1}{\phi_1\phi_2\cdots\phi_n}\cdot\frac{1}{\phi_i}.
  \]
  Then, the following metrics are projectively equivalent:
  \[
    g = \sum\Pi_iA_idx_i^2\qquad\qquad\tilde{g}=\pm\sum \rho_i\Pi_iA_idx_i^2
  \]
\end{prop}

\begin{rmk}\label{rmk:D-LC}
  We make a couple observations about these results:

  \begin{enumerate}
  \item Levi-Civita's result is actually more broad than presented here, and just like Dini's is actually an if and only if condition. See \cite{MT-TrajEquiv} for a modern exposition of the full result.

  \item While the results as originally stated are only about Riemannian metrics, they work for metrics of arbitrary signature. In the case of Levi-Civita's work, simply choose some of the $\phi_i$ to be negative. In particular, to get $(p,q)$ signature, take $p$ positive $\phi_i$'s and $q$ negative. If $q$ is odd, you also need to flip the sign on $\tilde{g}$. For the purposes of this paper, we can ignore the sign of $\tilde{g}$, since $SO^+(p,q) \cong SO^+(q,p)$.

  \item Dini's construction is a special case of Levi-Civita's, with the two being related by the following change of notation:
    \begin{align*}
      u &= x_2 & U_1 &= A_2 & U &= \phi_2\\
      v &= x_1 & V_1 &= A_1 & V &= \phi_1
    \end{align*}
  \end{enumerate}
\end{rmk}

Since Levi-Civita's result is for any dimension, it will be the one we use more frequently. The following proposition allows us to adapt Levi-Civita's work to the setting of manifolds with density:

\begin{prop}\label{prop:LC-W}
  Using the notation from Proposition~\ref{prop:LC}, let $\phi = \frac{1}{2}\log\left|\prod\phi_i\right|$, then the weighted connection of $(M,g,\phi)$ is the same as the Levi-Civita connection of $(M,\tilde{g})$.
\end{prop}

\begin{proof}
	We provide a combined proof of Propositions~\ref{prop:LC} and~\ref{prop:LC-W}.
	
	Let $A_i,\phi_i,\Pi_i,\rho_i,g,\tilde{g}$ be as above.
	
	We start by making a few observations:
	
	\begin{align*}
		g_{ij} &= \begin{cases}
			\Pi_i A_i & i=j\\
			0 & i\neq j
		\end{cases}\\
		\tilde{g}_{ij} &= \begin{cases}
			\rho_i\Pi_i A_i = \rho_ig_{ii} & i = j\\
			0 & i\neq j
		\end{cases}\\
		\partial_l g_{ii} &= \begin{cases}
			\frac{\left(\partial_l\phi_l\right)\Pi_i A_i}{\phi_l-\phi_i} & l\neq i\\
			\left(\partial_i A_i\right) \Pi_i + \sum_{k\neq i} \frac{\left(\partial_i\phi_i\right)\Pi_i A_i}{\phi_i-\phi_k} & l=i
		\end{cases}\\
		\partial_l \tilde{g}_{ii} &= \begin{cases}
			\rho_i\partial_l g_{ii} - \frac{\partial_l\phi_l}{\phi_l}\rho_ig_{ii} & l\neq i\\
			\rho_i\partial_i g_{ii} - \frac{2\partial_i\phi_i}{\phi_i}\rho_i g_{ii} & l=i
		\end{cases}
	\end{align*}
	
	Let
	\[\phi = \frac{1}{2}\log\left|\prod\phi_i\right| = \frac{1}{2}\sum\log|\phi_i|\]
	as in Proposition~\ref{prop:LC-W}
	
	Then, our goal is to show that the Christoffel symbols $\Gamma_{ij}^k$ of $g$ and $\tilde{\Gamma}_{ij}^k$ of $\tilde{g}$ are related by
	\[
	\tilde{\Gamma}_{ij}^k = \Gamma_{ij}^k - \left(\partial_i\phi\right)\delta_j^k - \left(\partial_j\phi\right)\delta_i^k
	\]
	that is, the Levi-Civita connection of $(M,\tilde{g})$ is equal to $\nabla^\phi$ on $(M,g,\phi)$. To achieve this, observe that
	\[
	\partial_i\phi = \frac{\partial_i\phi_i}{2\phi_i}
	\]
	
	First, observe that since $g,\tilde{g}$ are diagonal, if $i,j,k$ are all distinct, then $\Gamma_{ij}^k = \tilde{\Gamma}_{ij}^k = 0$. So, we only need to consider the cases when two or more of them are the same. We can break it up into 3 cases $i=j=k$, $i=j\neq k$ and $i=k\neq j$ (the $i\neq j=k$ is identical to the last one since the connections are torsion-free, and so need not be considered separately).
	
	First consider the case $i=j=k$. In this case:
	\begin{align*}
		\tilde{\Gamma}_{ii}^i &= \frac{1}{2}\tilde{g}^{ii}\partial_i\tilde{g}_{ii}\\
		&= \frac{1}{2}\cdot\frac{1}{\rho_i}g^{ii}\left(\rho_i\partial_ig_{ii} - 2\frac{\partial_i\phi_i}{\phi_i}\rho_ig_{ii}\right)\\
		&= \frac{1}{2}g^{ii}g\partial_i g_{ii} - \frac{\partial_i\phi_i}{\phi_i}\\
		&= \Gamma_{ii}^i - \left(\partial_i\phi\right) \delta_i^i - \left(\partial_i\phi\right)\delta_i^i
	\end{align*}
	
	Next, consider $i=k\neq j$. In this case:
	\begin{align*}
		\tilde{\Gamma}_{ij}^i &= \frac{1}{2}\tilde{g}^{ii}\partial_j\tilde{g}_{ii}\\
		&= \frac{1}{2}\cdot\frac{1}{\rho_i} g^{ii}\left(\rho_i\partial_j g_{ii} - \frac{\partial_j\phi_j}{\phi_j}\rho_ig_{ii}\right)\\
		&= \frac{1}{2} g^{ii}\partial_j g_{ii} - \frac{\partial_j\phi_j}{2\phi_j}\\
		&= \Gamma_{ij}^i - \left(\partial_i\phi\right) \delta_j^i - \left(\partial_j\phi\right)\delta_i^i
	\end{align*}
	
	Finally, consider $i=j\neq k$. We get:
	\begin{align*}
		\tilde{\Gamma}_{ii}^k &= \frac{-1}{2}\tilde{g}^{kk}\partial_k\tilde{g}_{ii}\\
		&= \frac{-1}{2}\cdot\frac{1}{\rho_k} g^{kk} \left(\rho_i\partial_k g_{ii} - \frac{\partial_k\phi_k}{\phi_k}\rho_i g_{ii}\right)\\
		&= \frac{-1}{2}g^{kk} \left(\frac{\phi_k}{\phi_i}\partial_k g_{ii} - \frac{\partial_k\phi_k}{\phi_i}g_{ii}\right)\\
		&= \frac{-1}{2}g^{kk}\left(\frac{\phi_k}{\phi_i}\cdot\frac{\left(\partial_k\phi_k\right)\Pi_i A_i}{\phi_k-\phi_i} - \frac{\partial_k\phi_k}{\phi_i}\Pi_i A_i\right)\\
		&= \frac{-1}{2}g^{kk} \frac{\left(\partial_k\phi_k\right)\Pi_i A_i}{\phi_k-\phi_i}\\
		&= \frac{-1}{2} g^{kk}\partial_k g_{ii}\\
		&= \Gamma_{ii}^k - \left(\partial_i\phi\right) \delta_i^k - \left(\partial_i\phi\right)\delta_i^k
	\end{align*}
\end{proof}

\subsection{Structural Result}\label{sec:Struct}

A useful tool in building high dimensional examples is the following result that describes how $\nabla^\phi$-holonomy interacts with totally geodesic submanifolds.

\begin{thm}\label{thm:totgeod}
  Let $(N,g,\phi)$ be a manifold with density. Let $M\subset N$ be a totally geodesic submanifold. Let $p\in M$, and let $\sigma:[0,T]\to M$ be a (piecwise-smooth) loop in $M$ based at $p$. Then,
  \[
    P_\sigma^{N,\phi} = \begin{pmatrix}
      P_\sigma^{M,\phi\mid_M} & X_\sigma\\
      0 & \left.P_\sigma^{N,g}\right|_{T_p^\perp M}
    \end{pmatrix}
  \]

  where $P_\sigma$ denotes parallel translation along $\sigma$, and the superscript denotes the manifold on which this parallel transport is taken, and the type of parallel transport (metric or density).

  The block $X_\sigma:T_p^\perp M\to T_p M$ is given by $X_\sigma(\vec{n}) = U(T)$, where $U$ is a vector field tangent to $M$ along $\sigma$, satisfying:
  \begin{align*}
    \nabla^{M,\phi|_M}_{\dot{\sigma}} U(t) &= e^{\phi(\sigma(t))-\phi(p)}d\phi(\vec{n}(t))\dot\sigma & U(0) &= 0\\
    \nabla_{\dot{\sigma}}^{N,g} \vec{n}(t) &= 0 &     \vec{n}(0) &= \vec{n}
  \end{align*}
\end{thm}

\begin{rmk}
  Of special interest to us in this paper will be the case where $\nabla \phi\in T_p M$ for all $p\in M$. In that case, it is easy to see that $X_\sigma=0$, since $d\phi(\vec{n}(t))=0$ along $M$.
\end{rmk}

\begin{proof}
  We prove this by considering two cases. One where our parallel vector field starts off tangent to $M$, and the other where our parallel vector field starts off orthogonal to $M$.

  Let $Y(t)$ be a $\nabla^{M,\phi\mid_M}$-parallel vector field along $\sigma(t)$, then
  \begin{align*}
    \nabla^{N,\phi}_{\dot\sigma} Y(t) &= \nabla^{N,g}_{\dot\sigma} Y(t) - d\phi(\dot\sigma)Y(t) - d\phi(Y(t))\dot\sigma\\
                                      &= \nabla^{M,g\mid_M}_{\dot\sigma} Y(t) - d\phi(\dot\sigma)Y(t) - d\phi(Y(t))\dot\sigma\\
                                      &= \nabla^{M,\phi\mid_M}_{\dot\sigma} Y(t) = 0,
  \end{align*}
  so $Y(t)$ is $\nabla^{N,\phi}$-parallel.

  Pick $\vec{n}\in T_p^\perp M$, and let $U(t)$ and $\vec{n}(t)$ be as in the statement of the theorem. Furthermore, let $\lambda(t) = e^{\phi(\sigma(t))-\phi(p)}$. We claim that $U(t) + \lambda(t)\vec{n}(t)$ is $\nabla^{N,\phi}$-parallel along $\sigma$.
  \begin{align*}
    \nabla^{N,\phi}_{\dot\sigma} \left( U(t) + \lambda(t)\vec{n}(t)\right) &= \nabla^{M,\phi|_M}_{\dot\sigma} U(t) + \lambda'(t)\vec{n}(t) + \lambda(t)\nabla^{N,g}_{\dot\sigma}\vec{n}(t)\\
    &\qquad\qquad\qquad- \lambda(t)d\phi(\dot\sigma)\vec{n}(t) - \lambda(t)d\phi(\vec{n}(t))\dot\sigma\\
                                                                           &=\nabla^{M,\phi\mid_M}_{\dot\sigma} U(t) - \lambda(t)d\phi(\vec{n}(t))\dot\sigma\\
    &= 0.
  \end{align*}
  Now, observe that $\lambda(T) = e^{\phi(p)-\phi(p)} = 1$, so starting with $\vec{n}$ we end at $X_\sigma(\vec{n}) + \vec{n}$, which completes the proof.
\end{proof}

This result is closely resembles \cite[Proposition 5.6]{WYDens}. However, it differs in that our assumptions are now local, and we also restrict the loop, whereas the referenced result makes a global structural assumption about the manifold, but allows for an arbitrary loop.


\section{Specific Examples}\label{sec:Ex}

In this section we provide specific examples of possible holonomy groups for manifolds with density. One goal of this is to show how different the $\nabla^\phi$ holonomy groups are from what can occur in the Riemannian case.

\subsection{Two-Dimensional Examples}\label{sec:2d}

We begin with the simplest case of $Hol^\phi$, namely the 2-dimensional setting. We prove the following:

\begin{repthm}{thm:2d}
  Every connected subgroup of $SL_2(\bbR)$ can be realized as the $\nabla^\phi$-holonomy of some $(M^2,g,\phi)$.
\end{repthm}

\begin{proof}
  Up to conjugation, $SL_2(\bbR)$ has 6 connected subgroups.

  \begin{enumerate}
  \item $SL_2(\bbR)$, we construct this example in Example~\ref{ex:SL2}

  \item The unique 2-dimensional connected subgroup, one of whose representations is
    \[
      \left\{\left.
          \begin{pmatrix}
            r & x\\
            0 & {1/r}
          \end{pmatrix}
        \right|
        x\in\bbR,\ r\in\bbR^+
      \right\}
    \]
    we construct such a holonomy group in Example~\ref{ex:2d}

  \item The one-dimensional connected subgroup, of the form
    \[
      \left\{\left.
          \begin{pmatrix}
            1 & x\\
            0 & 1
          \end{pmatrix}
        \right|
        x\in\bbR
      \right\}
    \]
    we refer to this subgroup as $\bbR$, and construct an example of it in Example~\ref{ex:R}

  \item $SO^+(1,1)$ we construct an example with such holonomy in Example~\ref{ex:SO+}

  \item $SO(2)$ occurs as a Riemannian holonomy, and as such occurs as a weighted holonomy with $\phi\equiv C$. (e.g. $(S^2,g_{round},0)$)
    
  \item $\{I\}$ occurs as a Riemannian holonomy, and as such occurs as a weighted holonomy with $\phi\equiv C$. (e.g. $(\bbR^2,g_{Eucl},0)$)
  \end{enumerate}
\end{proof}

\begin{ex}\label{ex:SL2}\cite[Example~5.10]{WYDens}
  Consider the round 2-sphere with density $(S^2,dr^2+\sin^2 r, \cos r)$. We consider two families of loops (each corresponding to a rectangle in local coordinates) on this sphere:
  
  $\alpha^s(t)$ with $s\in(0,\xi-\pi/2)$, corresponding to the path $(\pi/2,0) \to (\pi/2+s,0) \to (\pi/2+s,2\pi) \to (\pi/2,2\pi) \to (\pi/2,0)$. This is just the path traced out by starting at a point along the equator, moving up or down along a meridian, transiting once along the parallel, returning to the starting point along the meridian, and closing the loop along the equator. The closing along the equator makes it so that in the limit $s\to 0$ $\alpha^s(t)$ is a loop along the equator travelling once in one direction, and once in reverse, which makes the corresponding holonomy element the identity. (Note that returning along the equator is a small deviation from the the example in \cite{WYDens}, chosen for more explicit computation).

  Take parallel translation along $\alpha^s$, and differentiate it with respect to $s$ at $s=0$. Then, one obtains:
  \[
    \begin{pmatrix}
      2\pi^2 & -2\pi\\
      4\pi + \frac{8\pi^3}{3} & -2\pi^2
    \end{pmatrix} \in \mathfrak{hol}^\phi
  \]

  $\beta^s(t)$ with $s\in\bbR$, corresponding to the path $(\pi/2,0) \to (\xi,0) \to (\xi,s) \to (\pi/2,s) \to (\pi/2,0)$. This path travels from the equator along a meridian to a specific latitude, travels an arbitrary distance along that latitude, comes back to the equator and returns along the equator. Just as with $\alpha^0$, $\beta^0(t)$ is a loop with trivial holonomy, since it just goes up and down along the same path.

  The number $\xi = \cos^{-1}\left(\frac{1-\sqrt{5}}{2}\right)$ comes from specifics of parallel transport along lines of latitude along the sphere, in particular, there are three different possible behaviors base on the value of $r$: $r\in(\pi/2,\xi),\ r\in\{\pi/2,\xi\},\ r\not\in[\pi/2,\xi]$.

  Take parallel translation along $\beta^s$, and differentiate it with respect to $s$ at $s=0$. Then, one obtains:
  \[
    \begin{pmatrix}
      0 & \frac{1-\sqrt{5}}{2}e^{\frac{\sqrt{5}-1}{2}}\\
      1 & 0
    \end{pmatrix} \in \mathfrak{hol}^\phi
  \]

  One can check that the Lie algebra spanned by the two elements of $\mathfrak{hol}^\phi$ is $\mathfrak{sl}_2(\bbR)$, so the holonomy group must be $SL_2(\bbR)$.
\end{ex}

\begin{ex}\label{ex:2d}
  The next example we consider is the unique connected 2-dimensional subgroup of $SL_2(\bbR)$, consisting of matrices of the form
  \[
    \begin{pmatrix}
      r & x\\
      0 & 1/r
    \end{pmatrix}.
  \]
  Consider the manifold $(\bbR^2,dx^2+e^{2xy}dy^2,xy)$. Consider a loop $(a(t),b(t))$, and a $\nabla^\phi$-parallel vector field along it $u(t)\partial_x + v(t)\partial_y$. Then, the parallel transport ODEs are:
  \begin{align*}
    u'(t) &= \left(2b(t)a'(t)+a(t)b'(t)\right)u(t) + \left(a(t)a'(t)+b(t)b'(t)e^{2a(t)b(t)}\right)v(t)\\
    v'(t) &= a(t)b'(t)v(t)
  \end{align*}
  It is clear that $v(t)$ does not depend on $u(0)$, so the holonomy lies inside the desired group. It remains to show that the holonomy is 2-dimensional. For this, note that the Lie algebra $\fg$ corresponding to this Lie group has the property that $\exp:\fg\to G$ is bijective. So, in particular, we have an inverse:
  \[
    \exp^{-1}\begin{pmatrix}
      r & x\\
      0 & 1/r
    \end{pmatrix} = \begin{pmatrix}
      \log(r) & \frac{2x\log(r)}{r-1/r}\\
      0 & -\log(r)
    \end{pmatrix}
  \]
  This means that if we can find two elements of $Hol^\phi$, whose $\exp^{-1}$ images are linearly independent, we will have shown that the Lie group is as claimed.

  We consider two ``rectangular'' loops $(0,0)\to(1,0)\to(1,1)\to(0,1)\to(0,0)$ and $(0,0)\to(2,0)\to(2,1/2)\to(0,1/2)\to(0,0)$.

  The first one gives us
  \[
    \begin{pmatrix}
      e^{-1} & \frac{3-e^2}{2e}\\
      0 & e
    \end{pmatrix}\in Hol^\phi
    \qquad\Rightarrow\qquad
    \begin{pmatrix}
      -1 & \frac{3-e^2}{e^2-1}\\
      0 & 1
    \end{pmatrix}\in\mathfrak{hol}^\phi
  \]

  The second one gives us
  \[
    \begin{pmatrix}
      e^{-1} & \frac{81-17e^2}{16e}\\
      0 & e
    \end{pmatrix}\in Hol^\phi
    \qquad\Rightarrow\qquad
    \begin{pmatrix}
      -1 & \frac{81-17e^2}{8(e^2-1)}\\
      0 & 1
    \end{pmatrix}\in\mathfrak{hol}^\phi
  \]

  Since the two Lie algebra elements are clearly linearly independent, we have established that the holonomy group is as claimed.
\end{ex}

\begin{ex}\label{ex:R}
  Our next example is the one-dimensional subgroup
  \[
    \left\{\left.
        \begin{pmatrix}
          1 & x\\
          0 & 1
        \end{pmatrix}
      \right| x\in\bbR\right\}
  \]



  To construct this example, let $M = \bbR^2$, $g = e^xdx^2 + e^{2x+y}dy^2$, and $\phi = x+y$. Note, there are simpler examples (e.g. \cite[Example~5.7]{WYDens}), but this example provides a useful generalization to higher dimension, as we will see in Example~\ref{ex:Heis}. Consider a loop $\sigma(t)=(a(t),b(t))$, and a $\nabla^\phi$-parallel vector field along it $u(t)\partial_x + v(t)\partial_y$. Then, the parallel transport ODEs are
  \begin{align*}
    u'(t) &= \left(\frac{3}{2}a'(t)+b'(t)\right) u(t) + \left(a'(t)+e^{a(t)+b(t)}b'(t)\right) v(t)\\
    v'(t) &= \frac{3}{2}b'(t)v(t)
  \end{align*}
  Which give us the matrix
  \[
    \begin{pmatrix}
      1 & \Omega(\sigma)\\
      0 & 1
    \end{pmatrix},
  \]
  where $\Omega(\sigma)$ is a quantity that depends on the loop $\sigma$.

  It remains to show that $\Omega(\sigma)$ is not identically 0. Which follows from the fact that
  \[
    \Ric^\phi = \frac{1+e^{x+y}}{2}dy^2,
  \]
  so the connection is not flat, and the holonomy cannot be trivial.
\end{ex}

\begin{ex}\label{ex:SO+}
  The last two dimensional example is $SO^+(1,1)$. Here we use the work of Dini~\cite{Dini} and Levi-Civita~\cite{LC-Trasforma} we discussed in Section~\ref{sec:Opq}. Consider the example $M=\bbR^2$, $g= (3+\cos y)(dx^2+dy^2)$, and $\phi = \frac{1}{2}\log(2+\cos y)$. As stated before, the results in Section~\ref{sec:Opq} tell us that the resulting $\nabla^\phi$ is identical to the Levi-Civita connection on $(M,\tilde{g})$, where
  \[
    \tilde{g} = \frac{3+\cos y}{2+\cos y}dx^2 - \frac{3+\cos y}{(2+\cos y)^2}dy^2
  \]
  So, the holonomy lies inside $SO^+(1,1)$. Furthermore, at the origin, the Ricci curvature $\Ric^\phi = \frac{1}{8}dx^2 - \frac{1}{24}dy^2$, so the holonomy is non-tivial, and must be precisely $SO^+(1,1)$.
\end{ex}

\subsection{Higher-Dimensional Examples}\label{sec:hd}

In this section we discuss some higher dimensional examples, starting with dimension 3. In particular, we prove:

\begin{repthm}{thm:3d}
  Every 3-dimensional candidate given by Merkulov and Schwachh\"offer that lies inside $SL_3(\bbR)$ can arise as $Hol^\phi$.
\end{repthm}

According to the work of Merkulov and Schwachh\"ofer~\cite{MeSch}, there are 3 possible irreducible holonomy groups in 3 dimensions that lie inside $SL_3(\bbR)$: $SO(3),\ SO^+(2,1)$ and $SL_3(\bbR)$. We will show that all of these can occur as $\nabla^\phi$ holonomy groups. $SO(3)$ since it occurs as Riemmanian holonomy, and the other two because the full families $SO^+(p,q)$ and $SL_n(\bbR)$ occur.

We begin by providing the full details of a construction that was outlined in the author's paper with Wylie \cite{WYDens}, of how to obtain $Hol^\phi = SL_n(\bbR)$. Recall that this is the largest that a holonomy of an oriented manifold with density can possibly be, and as such can be viewed as a generic case.

\begin{lem}\label{lem:SkSn}
  Consider a round $S^n$, and a totally geodesic $S^k\subset S^n$. Let $\sigma:[0,T]\to S^k$ be a loop based at $p\in S^k$, then $\left.P_\sigma\right|_{T_p^\perp S^k} = I$.
\end{lem}

\begin{proof}
  This result is trivial when $k=n-1$, since both manifolds are orientable. Now, consider the chain of totally-geodesic embeddings $S^k\subset S^{k+1}\subset\cdots\subset S^{n-1}\subset S^n$. Since $P_\sigma$ is trivial on each $T_p^{\perp}S^l\subset T_p S^{l+1}$, it is trivial on $T_p^\perp S^k\subset T_p S^n$.
\end{proof}

\begin{prop}\label{prop:sln}
  The weighted holonomy group of $(S^n,dr^2 + \sin^2 r g_{S^{n-1}},\cos r)$ is $SL_n(\bbR)$.
\end{prop}

\begin{proof}
  We prove this by relying on the 2-dimensional case of $(S^2,dr^2+\sin^2 rd\theta^2,\cos r)$ proven in Example~\ref{ex:SL2}. The proof will be based on Theorem~\ref{thm:totgeod}, and be carried out on the Lie algebra level.

  Let $p$ be any point on $S^n$ with $r=\pi/2$. (i.e. on the ``equator''). Pick a coordinate system $(x_1,x_2,\ldots,x_n)$ on a neighborhood of $p$, such that at $p$ $\partial_{x_1} = \partial_r$ and $\partial_{x_i}$ are orthonormal. Consider the totally geodesic $S^2$ that contains the point $p$ and whose tangent space at $p$ is spanned by $\partial_{x_1},\partial_{x_i}$ for some $i$.

  Using Lemma~\ref{lem:SkSn} and Theorem~\ref{thm:totgeod}, we know that $Hol^{S^n,\phi}_p$ contains as a subgroup $SL_2(\bbR)$ corresponding to $Hol^{S^2,\phi}_p$. In particular, if we use the notation $E_{ij}$ for the matrix with 1 in the $(i,j)$ entry and 0s everywhere else, then we have established that $\mathfrak{hol}^{S^n,\phi}_p$ contains $E_{1i},E_{i1},E_{ii}-E_{11}$. However, $i$ was arbitrary, so this is true for every $i$.

  Now, consider the basis $E_{ij}$ ($i\neq j$), $E_{11}-E_{ii}$ for $\mathfrak{sl}_n(\bbR)$. We have established that every element of the second family is in $\mathfrak{hol}^{S^n,\phi}_p$. For $E_{ij}$, pick $i\neq j$, and note that $[E_{i1},E_{1j}] = E_{ij}$, therefore, this family is also included. Thus we have established that $\mathfrak{hol}^{S^n,\phi}_p \supseteq \sl_n(\bbR)$, but by \cite{WYDens}, the holonomy is no larger than $SL_n(\bbR)$. Therefore, it must be exactly $SL_n(\bbR)$.
\end{proof}

The other infinite family of examples we want to construct are the ones that $Hol^\phi = SO^+(p,q)$, where $\dim M = p+q$. As mentioned previously, the key tool for these constructions is the result of Levi-Civita, since we are dealing with a pseudo-Riemannian holonomy group.

\begin{prop}\label{prop:sopq}
  Consider $\bbR^n$ constructed in accordance with the result of Levi-Civita (see Propositions~\ref{prop:LC} and \ref{prop:LC-W}), where
  \[
    \phi_n = n + \cos x_n
  \]
  and the other $\phi_i$ are constant with $p$ negative values, and $q-1$ positive values, and $A_i=1$. The resulting weighted holonomy is precisely $SO^+(p,q)$.
\end{prop}

\begin{lem}\label{lem:sopq2d}
  Let $(N^n,g,\phi)$ be the manifold described in Proposition~\ref{prop:sopq}, and let $M_i^2\subset N$ be the two-dimensional submanifold given by $x_j=0$ for $j\not\in\{i,n\}$. Then the $\nabla^{\phi\mid_{M_i}}$-holonomy group is $SO(2)$ or $SO^+(1,1)$.
\end{lem}

\begin{proof}[Proof of Lemma~\ref{lem:sopq2d}]
  Observe that we can view $(M_i^2,g\mid_{M_i})$ as itself being a Levi-Civita example, by letting $A_1 = \frac{\Pi_i}{\phi_n-\phi_i},\ \phi_1 = \phi_i$ and $A_2 = \frac{\Pi_n}{\phi_n-\phi_i},\ \phi_2 = \phi_n$. Furthermore, $\phi\mid_{M_i}$ is up to an additive constant the same as the density function given by Proposition~\ref{prop:LC-W}, and the additive constant does not impact the connection. Therefore, the holonomy is either one of the claimed ones or trivial. Explicit computation shows that $\Ric^\phi$ is non-trivial. In particular,
  \[
    \Ric^\phi(\partial_i,\partial_i) = \frac{-\phi_i\Pi_i}{2\prod_{j=1}^n (n+1-\phi_j)} \neq 0
  \]
  whenever $x_n=0$. Therefore, the holonomy group can't be trivial, and must be $SO(2)$ or $SO^+(1,1)$, depending on the sign of $\phi_i$.
\end{proof}

\begin{lem}\label{lem:parsopq}
  Let $M_i$ and $N$ be as in Lemma~\ref{lem:sopq2d}, then, $M_i\subset N$ is a totally geodesic submanifold, and $g$-parallel translation along any loop $\sigma$ in $M_i$ based at $p$ leaves $T^\perp_p M_i$ fixed.
\end{lem}

\begin{proof}[Proof of Lemma~\ref{lem:parsopq}]
  This result is entirely analogous to Lemma~\ref{lem:SkSn}. We can construct a chain of embeddings $M_i = X_2 \subset X_3 \subset \cdots\subset X_n = N$, each $X_i$ including one more coordinate direction from $N$. Let $g_k = g\mid_{X_k}$, then, $g_{k+1} = g_k + u(x_n)dx_{t_{k+1}}^2$, where $x_{t_{k+1}}$ denotes the coordinate added in $X_{k+1}$. Since each metric is a warped product of the previous one, each $X_k$ is therefore totally geodesic inside $X_{k+1}$. For parallel translation, observe that viewing $\sigma$ as being a loop in $X_k$, means that $x_{t_{k+1}}$ direction is preserved. Therefore, the translation acts trivially on $T^\perp_p M_i$.
\end{proof}

\begin{proof}[Proof of Proposition~\ref{prop:sopq}]
  From the work of Levi-Civita, we know that the holonomy is a subgroup of $SO^+(p,q)$.
  
  Using Lemmas~\ref{lem:sopq2d} and \ref{lem:parsopq}, we conclude that $E_{in}+ \epsilon_i E_{ni}$ is in $\mathfrak{hol}^\phi$ for every $i\in\{1,\ldots,n-1\}$, with $\epsilon_i=1$ if $M_i$ has $SO^+(1,1)$ holonomy, and $\epsilon_i=-1$ if $M_i$ has $SO(2)$ holonomy.

  Now, consider $i\neq j$, then
  \[[E_{in}+\epsilon_i E_{ni}, E_{jn}+\epsilon_j E_{nj}] = \epsilon_j E_{ij} - \epsilon_i E_{ji}\in\mathfrak{hol}^\phi\]
  Therefore, the holonomy Lie algebra is all of $\mathfrak{so}^+(p,q)$. So, the group is $SO^+(p,q)$ as claimed.
\end{proof}


\begin{ex}\label{ex:Heis}
  One other 3 dimensional example we wish to present is a manifold with Heisenberg group as its $\nabla^\phi$-holonomy. Let $M = \bbR^3$ with $g = e^{x}dx^2 + e^{2x+y}dy^2 + e^{2x+2y+z}dz^2$, and $\phi = x+y+z$.

  Direct observation of the parallel transport ODEs shows that the holonomy group is no larger than the Heisenberg group. Explicit computations show that the holonomy Lie algebra at the origin ($\mathfrak{hol}^\phi_{0}$) contains an element of the form
  \[
    A = \begin{pmatrix}
      0 & 1 & x\\
      0 & 0 & 0\\
      0 & 0 & 0
    \end{pmatrix}
  \]
  This element is obtained by traversing the piece-wise linear loop $(0,0,0)\to(1,0,0)\to(1,1,0)\to(0,1,0)\to(0,0,0)$, taking the logarithm of the resulting element of $Hol^\phi$ and re-scaling it.

  To obtain the rest of the Heisenberg Lie algebra, we make an observation that this example is self-dual. In the sense that there is a diffeomorphism $F:M\to M$ such that $F^* g = \tilde{g} = e^{-2\phi}g$ and $F^*\phi = -\phi$. Namely $F(x,y,z) = (-z,-y,-x)$; moreover, $F$ preserves the origin. This gives us two isomorphisms between the holonomy groups $Hol^{g,\phi}$ and $Hol^{\tilde{g},-\phi}$ at the origin, namely, the one induced by $F$, and the one given in Corollary~\ref{cor:DualHol}.

  The one induced by $F$ is obtained simply by conjugating by
  \[
    dF_{0} = \begin{pmatrix}
      0 & 0 & -1\\
      0 & -1 & 0\\
      -1 & 0 & 0
    \end{pmatrix}
  \]
  and the one from Corollary~\ref{cor:DualHol} is obrained by taking inverse transpose (since at the origin $e^{-\phi}g = dx^2+dy^2+dz^2$). Operating on the Lie algebra level, the isomorphism induced by $F$ is still the conjugation by $dF_0$, but the duality isomorphism becomes $A\to -A^T$. We know compose these isomorphisms.

  Since
  \[
    A = \begin{pmatrix}
      0 & 1 & x\\
      0 & 0 & 0\\
      0 & 0 & 0
    \end{pmatrix} \in \mathfrak{hol}^{g,\phi}_0
  \]
  conjugation by $dF_0$ gives us
  \[
    \begin{pmatrix}
      0 & 0 & 0\\
      0 & 0 & 0\\
      x & 1 & 0
    \end{pmatrix} \in\mathfrak{hol}^{\tilde{g},-\phi}_0
  \]
  and then duality gives us
  \[
    B = \begin{pmatrix}
      0 & 0 & -x\\
      0 & 0 & -1\\
      0 & 0 & 0
    \end{pmatrix} \in\mathfrak{hol}^{g,\phi}_0.
  \]
  Now, observe that
  \[
    [B,A] = \begin{pmatrix}
      0 & 0 & 1\\
      0 & 0 & 0\\
      0 & 0 & 0
    \end{pmatrix}
  \]
  Therefore, $A,B,[B,A]$ span the entire Heisenberg Lie algebra, and so $Hol^{g,\phi}_0$ is precisely the Heisenberg group.
\end{ex}

This example together with its 2-dimensional analogue in Example~\ref{ex:R} lead us to the following conjecture:

\begin{conj}
  The holonomy Lie algebra of $(M^n,g,\phi)$ computed at $(0,0,\ldots,0)$ where $\phi=x_1+x_2+\cdots+x_n$, and $g = e^{x_1}dx_1^2 + e^{2x_1+x_2}dx_2^2 + \cdots + e^{2x_1+\cdots+2x_{n-1}+x_n}dx_n^2$ is precisely the strictly upper triangular Lie algebra.
\end{conj}

We have now seen this for $n=2,3$. The self-duality observation used in Example~\ref{ex:Heis} is still valid in higher dimensions. Additionally, one may be able to use Theorem~\ref{thm:totgeod}, since $(x_1,\ldots,x_k,0,\ldots,0)$ is a $k$-dimensional totally-geodesic submanifold. However, one would need to show certain properties of $X_\sigma$ to make this easy to use. In Example~\ref{ex:Heis}, the key was finding a loop such that the resulting Lie algebra element was not one of
\[
  \begin{pmatrix}
    0 & 1 & x\\
    0 & 0 & -1\\
    0 & 0 & 0
  \end{pmatrix},\qquad
  \begin{pmatrix}
    0 & 1 & x\\
    0 & 0 & 1\\
    0 & 0 & 0
  \end{pmatrix},\qquad
  \begin{pmatrix}
    0 & 0 & 1\\
    0 & 0 & 0\\
    0 & 0 & 0
  \end{pmatrix}.
\]
Since such elements would only generate a 1 or 2 dimensional subalgebra of the Heisenberg Lie algebra under the the composition of the duality and $dF_0$ isomorphisms.


\appendix
\section{Explicit Holonomy Computations}\label{sec:Compute}

We provide full computations for all the explicit examples here. For each example we provide the weighted Christoffel symbols ($\tilde{\Gamma}_{ij}^k$), the system or systems of ODEs, their solutions (or partial solutions if sufficient), and the resulting matrices. Note that if the path used is piece-wise smooth, then there will be a system if ODEs for each smooth piece. The purpose of most of these computations is to show that the resulting holonomy group is no smaller than claimed, for the cases where other techniques are insufficient. In other cases, partial solutions to translation along arbitrary loops will provide an upper bound on the size of the holonomy group. The notation throughout this appendix will be $X_u$ for the $u$ component of the parallel vector field along the given path.

\begin{repex}{ex:SL2}
  We consider $(S^2,r^2+\sin^2 r d\theta^2, \cos r)$ as our manifold with density. Then, the non-zero weighted Christoffel symbols are as follows:
  \begin{align*}
    \tilde{\Gamma}_{rr}^r &= 2\sin r\\
    \tilde{\Gamma}_{r\theta}^\theta &= \cot r + \sin r\\
    \tilde{\Gamma}_{\theta\theta}^r &= -\cos r\sin r
  \end{align*}

  We begin by considering the family of loops $\alpha^s(t)$. Since the loop consists of four smooth curves, we will have four systems of ODEs:

  Along the path $(\pi/2,0)\to (\pi/2+s,0)$, consider the parametrization $(\pi/2+t,0)$ $t\in[0,s]$. Then the system of ODEs is:

  \begin{align*}
    2\sin(\pi/2+t)X_r(t) + X'_r(t) &= 0\\
    \left(\cot(\pi/2+t) + \sin(\pi/2+t)\right)X_\theta(t) + X'_\theta(t) &= 0
  \end{align*}

  Solving, we get:
  \begin{align*}
    X_r(t) &= Ae^{-2\sin t} & X_r(s) &= X_r(0) e^{-2\sin s}\\
    X_\theta(t) &= B\sec t e^{-\sin t} & X_\theta(s) &= X_\theta(0) \sec s e^{-\sin s}
  \end{align*}

  Which corresponds to the matrix
  \[
    \begin{pmatrix}
      e^{-2\sin s} & 0\\
      0 & \sec s\ e^{-\sin s}
    \end{pmatrix}
  \]

  The second component of $\alpha^s(t)$ is the path $(\pi/2+s,0) \to (\pi/2+s,2\pi)$, which we parametrize as $(\pi/2+s,t)$ with $t\in[0,2\pi]$. Which gives us
  \begin{align*}
    -\cos(\pi/2+s)\sin(\pi/2+s)X_\theta(t) + X'_r(t) &= 0\\
    \left(\cot(\pi/2+s) + \sin(\pi/2+s)\right) X_r(t) + X'_\theta(t) &= 0
  \end{align*}

  Let $k = \sqrt{\sin s\cos^2 s - \sin^2 s}$ (since $s\in(0,\xi-\pi/2)$, this is a real number). Solving, we get:
  \begin{align*}
    X_r(t) &= (Ae^{-kt}+Be^{kt})\sin s\cos s = (\alpha\cosh(kt) - \beta\sinh(kt))\sin s\cos s\\
    X_\theta(t) &= (Ae^{-kt}-Be^{kt})k = (\beta\cosh(kt)-\alpha\sinh(kt))k\\
  \end{align*}
  Going from $t=0$ to $t=2\pi$, this corresponds to the matrix
  \[
    \begin{pmatrix}
      \cosh(2\pi k) & \frac{-\sin s\cos s\sinh(2\pi k)}{k}\\
      \frac{-k\sinh(2\pi k)}{\sin s\cos s} & \cosh(2\pi k)
    \end{pmatrix}
  \]

  The third component is precisely the reverse of the first, so the matrix is the inverse of the first.
  \[
    \begin{pmatrix}
      e^{2\sin s} & 0\\
      0 & \cos s\ e^{\sin s}
    \end{pmatrix}
  \]

  The final component is similar to the second. Parametrize it as $(pi/2,2\pi-t)$, then we get
  \begin{align*}
    X_r'(t) &= 0\\
    -X_r(t) + X_\theta(t) &= 0
  \end{align*}
  solving we get
  \begin{align*}
    X_r(t) &= A & X_r(2\pi) &= X_r(0)\\
    X_\theta(t) &= At + B & X_\theta(2\pi) &= 2\pi X_r(0) + X_\theta(0)
  \end{align*}
  which corresponds to the matrix
  \[
    \begin{pmatrix}
      1 & 0\\
      2\pi & 1
    \end{pmatrix}
  \]

  Multiplying all these together, we get that the parallel translation along $\alpha^s(t)$ is given by:
  \[
    \begin{pmatrix}
      \cosh(2\pi k) & \frac{-\sin s\sinh(2 \pi k)e^{\sin s}}{k}\\
      \frac{-k\sinh(2\pi k)e^{-\sin s}}{\sin s} + 2\pi\cosh(2\pi k)& \frac{-2\pi\sin s\sinh(2\pi k)e^{\sin s}}{k} + \cosh(2\pi k)
    \end{pmatrix}
  \]

  Let $P(s)$ denote the above parallel translation. Since $P(s)\to I$ as $s\to 0$, we evaluate
  \[
    \lim_{s\to 0} \frac{1}{s}\left[P(s)-I\right] = \begin{pmatrix}
      2\pi^2 & -2\pi\\
      4\pi + \frac{8\pi^3}{3}& -2\pi^2
    \end{pmatrix} \in \mathfrak{hol}^\phi
  \]

  The limits were found by factoring, and computing Taylor series of numerator and denominator.

  We next consider the family of loops $\beta^s(t)$. Once again, the loop consists of four smooth curves, so we will have four systems of ODEs:

  Along the path $(\pi/2,0) \to (\xi,0)$, we can use the same parametrization as with the first segment of $\alpha$, to get
  \[
    \begin{pmatrix}
      e^{1-\sqrt{5}} & 0\\
      0 & \frac{1+\sqrt{5}}{2}e^{\frac{1-\sqrt{5}}{2}}
    \end{pmatrix}
  \]

  Along the path $(\xi,0) \to (\xi,s)$, we can have the same system as with second piece of $\alpha$. However, $\xi$ was chosen precisely, so that $\cot(\xi) + \sin(\xi) = 0$. Which means that our system degenerates into:
  \begin{align*}
    X_r'(t) &= \sin\xi\cos\xi X_\theta(t)\\
    X_\theta'(t) &= 0
  \end{align*}
  which is easily solved as
  \begin{align*}
    X_r(t) &= X_r(0) + X_\theta(0)t\sin\xi\cos\xi\\
    X_\theta(t) &= X_\theta(0)
  \end{align*}
  giving us the matrix
  \[
    \begin{pmatrix}
      1 & -s \left(\frac{\sqrt{5}-1}{2}\right)^{3/2}\\
      0 & 1
    \end{pmatrix}
  \]

  The third segment gives us the inverse of the first matrix, since our example is fully rotationally symmetric.

  The fourth segment gives us the same system of equations as the fourth segment of $\alpha$, but since we only traverse from $(\pi/2,s)\to(\pi/2,0)$, the resulting matrix is
  \[
    \begin{pmatrix}
      1 & 0\\
      0 & s
    \end{pmatrix}
  \]

  Multiplying the four matrices together, we get
  \[
    P(s) = \begin{pmatrix}
      1 & s\left(\frac{1-\sqrt{5}}{2}\right)e^{\frac{\sqrt{5}-1}{2}}\\
      s & 1 + s^2\left(\frac{1-\sqrt{5}}{2}\right)e^{\frac{\sqrt{5}-1}{2}}
    \end{pmatrix}
  \]
  Differentiating at $s=0$, we get
  \[
    \begin{pmatrix}
      0 & \left(\frac{1-\sqrt{5}}{2}\right)e^{\frac{\sqrt{5}-1}{2}}\\
        1 & 0
    \end{pmatrix} \in \mathfrak{hol}^\phi
  \]
\end{repex}

\begin{repex}{ex:2d}
  We consider $(\bbR^2,dx^2 + e^{2xy}dy^2, xy)$ as our manifolds with density. Then, the non-zero weighted Christoffel symbols are as follows:
  \begin{align*}
    \tilde{\Gamma}_{xx}^x &= -2y\\
    \tilde{\Gamma}_{xy}^x &= -x\\
    \tilde{\Gamma}_{yy}^x &= -y e^{2xy}\\
    \tilde{\Gamma}_{yy}^y &= -x
  \end{align*}

  We begin with the first loop under consideration: $(0,0)\to(1,0)\to(1,1)\to(0,1)\to(0,0)$.

  We look at the first segment, parametrized as $(t,0)$ with $t\in[0,1]$. Then, the resulting system of ODEs is
  \begin{align*}
    X_x'(t) &= t X_y(t)\\
    X_y'(t) &= 0
  \end{align*}
  Solving, we get
  \begin{align*}
    X_x(t) &= X_x(0) + \frac{t^2}{2}X_y(0)\\
    X_y(t) &= X_y(0)
  \end{align*}
  which gives us the matrix
  \[
    \begin{pmatrix}
      1 & \frac{1}{2}\\
      0 & 1
    \end{pmatrix}
  \]

  Consider $(1,t)$ with $t\in[0,1]$ on the second segment. Then the resulting system is
  \begin{align*}
    X_x'(t) &= X_x(t) + te^{2t}X_y(t)\\
    X_y'(t) &= X_y(t)
  \end{align*}
  Solving, we get
  \begin{align*}
    X_x(t) &= A\left(\frac{t}{2}-\frac{1}{4}\right)e^{3t} + B e^t & X_x(t) &= X_x(0)e^t + X_y(0)\left[\left(\frac{t}{2}-\frac{1}{4}\right)e^{3t} + \frac{1}{4}e^t\right]\\
    X_y(t) &= A e^t & X_y(t) &= X_y(0)e^t
  \end{align*}
  giving us the matrix
  \[
    \begin{pmatrix}
      e & \frac{e^3+e}{4}\\
      0 & e
    \end{pmatrix}
  \]

  On the third segment, we consider the parametrization $(1-t,1)$, to get
  \begin{align*}
    X_x'(t) &= -2X_x(t) + (t-1)X_y(t)\\
    X_y'(t) &= 0
  \end{align*}
  solving we get
  \begin{align*}
    X_x(t) &= A\frac{2t-3}{4} + Be^{-2t} & X_x(t) &= X_x(0)e^{-2t} + X_y(0)\left(\frac{2t-3}{4} + \frac{3}{4}e^{-2t}\right)\\
    X_y(t) &= A & X_y(t) &= X_y(0)
  \end{align*}
  giving us the matrix
  \[
    \begin{pmatrix}
      e^{-2} & \frac{3e^{-2}-1}{4}\\
      0 & 1
    \end{pmatrix}
  \]

  Along the last segment, we consider the parametrization $(0,1-t)$, to get
  \begin{align*}
    X_x'(t) &= (t-1)X_y(t)\\
    X_y'(t) &= 0
  \end{align*}
  solving we get
  \begin{align*}
    X_x(t) &= X_x(0) + X_y(0)\left(\frac{t^2}{2}-t\right)\\
    X_y(t) &= X_y(0)
  \end{align*}
  wich gives us the matrix
  \[
    \begin{pmatrix}
      1 & -1/2\\
      0 & 1
    \end{pmatrix}
  \]
  
  Multiplying the four matrices together, we get
  \[
    P = \begin{pmatrix}
      e^{-1} & \frac{3-e^2}{2e}\\
      0 & e
    \end{pmatrix}
  \]

  The corresponding Lie algebra element then is
  \[
    \begin{pmatrix}
      -1 & \frac{3-e^2}{e^2-1}\\
      0 & 1
    \end{pmatrix}\in\mathfrak{hol}^\phi
  \]

  We now consider the second loop: $(0,0)\to(2,0)\to(2,1/2)\to(0,1/2)\to(0,0)$.

  The first segment gives us the same ODEs as the first segment of the first loop, so we get the matrix
  \[
    \begin{pmatrix}
      1 & 2\\
      0 & 1
    \end{pmatrix}
  \]
  
  For the second segment, we use the parametrization $(2,t)$ with $t\in[0,1/2]$, to get the system
  \begin{align*}
    X_x'(t) &= 2X_x(t)+te^{4t}X_y(t)\\
    X_y'(t) &= 2X_y(t)
  \end{align*}
  solving we get
  \begin{align*}
    X_x(t) &= A\left(\frac{t}{4}-\frac{1}{16}\right)e^{6t} + Be^{2t} & X_x(t) &= X_x(0)e^{2t} + X_y(0)\left[\left(\frac{t}{4}-\frac{1}{16}\right)e^{6t} + \frac{1}{16}e^{2t}\right]\\
    X_y(t) &= Ae^{2t} & X_y(t) &= X_y(0)e^{2t}
  \end{align*}
  giving us the matrix
  \[
    \begin{pmatrix}
      e & \frac{e^3+e}{16}\\
      0 & e
    \end{pmatrix}
  \]

  Parametrize the third segment as $(2-t,1/2)$ with $t\in[0,2]$, to get
  \begin{align*}
    X_x'(t) &= -X_x(t)+(t-2)X_y(t)\\
    X_y'(t) &= 0
  \end{align*}
  which we solve to get
  \begin{align*}
    X_x(t) &= C(t-3) + De^{-t} & X_x(t) &= X_x(0)e^{-t} + X_y(0)\left(t-3+3e^{-t}\right)\\
    X_y(t) &= C & X_y(t) &= X_y(0)
  \end{align*}
  giving us the matrix
  \[
    \begin{pmatrix}
      e^{-2} & 3e^{-2}-1\\
      0 & 1
    \end{pmatrix}
  \]

  We parametrize the last segment as $(0,1/2-t)$ with $t\in[0,1/2]$, to get
  \begin{align*}
    X_x'(t) &= (t-1/2)X_y(t)\\
    X_y'(t) &= 0
  \end{align*}
  Solving to get
  \begin{align*}
    X_x(t) &= X_x(0) + X_y(0)\left(\frac{t^2-t}{2}\right)\\
    X_y(t) &= X_y(0)
  \end{align*}
  which gives us the matrix
  \[
    \begin{pmatrix}
      1 & -1/8\\
      0 & 1
    \end{pmatrix}
  \]

  Multiplying the four together, we get:
  \[
    \begin{pmatrix}
      e^{-1} & \frac{81-17e^2}{16e}\\
      0 & e
    \end{pmatrix}
  \]
  which gives us
  \[
    \begin{pmatrix}
      -1 & \frac{81-17e^2}{8(1-e^2)}\\
      0 & 1
    \end{pmatrix}\in\mathfrak{hol}^\phi
  \]
\end{repex}

\begin{repex}{ex:R}
  



  We are considering the manifold $(M,g,\phi)$ with $M = \bbR^2$, $g = e^x dx^2 + e^{2x+y}dy^2$, and $\phi = x+y$.

  We can compute the non-zero weighted Christoffel symbols:
  \begin{align*}
    \tilde{\Gamma}_{xx}^x &= \frac{-3}{2} & \tilde{\Gamma}_{xy}^x &= -1\\
    \tilde{\Gamma}_{yy}^x &= -e^{x+y} & \tilde{\Gamma}_{yy}^y &= \frac{-3}{2}
  \end{align*}

  Consider a loop $\sigma(t)=(a(t),b(t))$, and a $\nabla^\phi$-parallel vector field along it $u(t)\partial_x + v(t)\partial_y$. Then, the parallel transport ODEs are
  \begin{align*}
    u'(t) &= \left(\frac{3}{2}a'(t)+b'(t)\right) u(t) + \left(a'(t)+e^{a(t)+b(t)}b'(t)\right) v(t)\\
    v'(t) &= \frac{3}{2}b'(t)v(t)
  \end{align*}
  solving, we get:
  \begin{align*}
    u(t) &= A e^{\frac{3}{2}a(t) + b(t)}\int_0^t \left(a'(\tau)e^{\frac{-3a(\tau)+b(\tau)}{2}} + b'(\tau)e^{\frac{-a(\tau)+3b(\tau)}{2}}\right)d\tau + Be^{\frac{3}{2}a(t)+b(t)}\\
    v(t) &= Ae^{\frac{3}{2}b(t)}
  \end{align*}

  Going over a loop $\sigma:[0,T]\to M$ based at $(0,0)\in M$, we get
  \[
    \begin{pmatrix}
      1 & \Omega(\sigma)\\
      0 & 1
    \end{pmatrix}
  \]
  where
  \[
    \Omega(\sigma) = \int_0^T\left(a'(\tau)e^{\frac{-3a(\tau)+b(\tau)}{2}} + b'(\tau)e^{\frac{-a(\tau)+3b(\tau)}{2}}\right)d\tau
  \]
\end{repex}

\begin{repex}{ex:SO+}
  We are considering $M=\bbR^2$, $g = (3+\cos x)(dx^2+dy^2)$, $\phi = \frac{1}{2}\log(2+\cos x)$.

  We can compute the non-zero weighted Christoffel symbols:
  \begin{align*}
    \tilde{\Gamma}_{xx}^y &= \frac{\sin y}{2(3+\cos y)}\\
    \tilde{\Gamma}_{xy}^x &= \frac{\sin y}{2(2+\cos y)(3+\cos y)}\\
    \tilde{\Gamma}_{yy}^y &= \frac{\sin y(4+\cos y)}{2(2+\cos y)(3+\cos y)}
  \end{align*}

  This matches precisely to the Levi-Civita connection of
  \[
    \tilde{g} = \frac{3+\cos y}{2+\cos y}dx^2 - \frac{3+\cos y}{(2+\cos y)^2}dy^2.
  \]
\end{repex}

\begin{repex}{ex:Heis}
  Consider $M=\bbR^3$ with $g = e^{x}dx^2 + e^{2x+y}dy^2 + e^{2x+2y+z}dz^2$, and $\phi = x+y+z$. The non-zero weighted Christoffel symbols are as follows:

  \begin{align*}
    \tilde{\Gamma}_{xx}^x &= \frac{-3}{2} & \tilde{\Gamma}_{xy}^x &= -1 & \tilde{\Gamma}_{xz}^x &= -1\\
    \tilde{\Gamma}_{yy}^x &= -e^{x+y} & \tilde{\Gamma}_{yy}^y &= \frac{-3}{2} & \tilde{\Gamma}_{yz}^y &= -1\\
    \tilde{\Gamma}_{zz}^x &= -e^{x+2y+z} & \tilde{\Gamma}_{zz}^y &= -e^{y+z} & \tilde{\Gamma}_{zz}^z &= \frac{-3}{2}
  \end{align*}

  Consider a loop $\sigma(t) = (\alpha_1(t),\alpha_2(t),\alpha_3(t))$, and a $\nabla^\phi$-parallel vector field along it $u(t)\partial_x + v(t)\partial_y + w(t)\partial_z$. Then, the parallel transport ODEs are:
  \begin{align*}
    u'(t) &= \left(\frac{3}{2}\alpha_1'(t) + \alpha_2'(t) + \alpha_3'(t)\right)u(t) + \left(\alpha_1'(t)+e^{\alpha_1(t)+\alpha_2(t)}\alpha_2'(t)\right)v(t)\\
          &\qquad\qquad+ \left(\alpha_1'(t)+e^{\alpha_1(t)+2\alpha_2(t)+\alpha_3(t)}\alpha_3'(t)\right)w(t)\\
    v'(t) &= \left(\frac{3}{2}\alpha_2'(t)+\alpha_3'(t)\right)v(t) + \left(e^{\alpha_2(t)+\alpha_3(t)}\alpha_3'(t)+\alpha_2'(t)\right)w(t)\\
    w'(t) &= \frac{3}{2}\alpha_3'(t)w(t)
  \end{align*}

  It is easy to see that the holonomy group is at most the Heisenberg group, since $w(t)$ depends only on the position, not the path, and if $w(t)=0$, then so does $v(t)$, and with $v(t)=w(t)=0$, $u(t)$ also only depends on the position. Therefore, we now only need to verify that we are able to obtain the entire Heisenberg group.

  We first consider the loop $(0,0,0)\to(1,0,0)\to(1,1,0)\to(0,1,0)\to(0,0,0)$ by ``straight'' lines.

  On the first segment, which we paramatrize as $(t,0,0)$, the ODEs become:
  \begin{align*}
    u'(t) &= \frac{3}{2}u(t) + v(t) + w(t)\\
    v'(t) &= 0\\
    w'(t) &= 0
  \end{align*}
  which gives us
  \begin{align*}
    u(t) &= u(0)e^{3t/2} + \frac{2}{3}\left(e^{3t/2}-1\right)v(0) + \frac{2}{3}\left(e^{3t/2}-1\right)w(0)\\
    v(t) &= v(0)\\
    w(t) &= w(0)
  \end{align*}
  giving us the matrix
  \[
    \begin{pmatrix}
      e^{3/2} & \frac{2}{3}e^{3/2}-\frac{2}{3} & \frac{2}{3}e^{3/2}-\frac{2}{3}\\
      0 & 1 & 0\\
      0 & 0 & 1
    \end{pmatrix}
  \]

  On the second segment, our parametrization is $(1,t,0)$, which gives us
  \begin{align*}
    u'(t) &= u(t) + e^{t+1}v(t)\\
    v'(t) &= \frac{3}{2}v(t) + w(t)\\
    w'(t) &= 0
  \end{align*}
  solving we get
  \begin{align*}
    u(t) &= Ae^{t} +\frac{2e}{3}B e^{5t/2} - \frac{2e}{3}Cte^{t}\\
    v(t) &= Be^{3t/2} - \frac{2}{3}C\\
    w(t) &= C
  \end{align*}
  giving us
  \[
    \begin{pmatrix}
      e & \frac{2}{3}e^2 - \frac{2}{3}e^{7/2} & \frac{4}{9}e^{7/2}-\frac{10}{9}e^2\\
      0 & e^{3/2} & \frac{2}{3}e^{3/2} - \frac{2}{3}\\
      0 & 0 & 1
    \end{pmatrix}
  \]

  Parametrize the third segment as $(1-t,1,0)$ to get
  \begin{align*}
    u'(t) &= \frac{-3}{2}u(t) - v(t) - w(t)\\
    v'(t) &= 0\\
    w'(t) &= 0
  \end{align*}
  Solving we get
  \begin{align*}
    u(t) &= e^{-3t/2}u(0) + \frac{2}{3}\left(e^{-3t/2}-1\right) v(0) + \frac{2}{3}\left(e^{-3t/2}-1\right)w(0)\\
    v(t) &= v(0)\\
    w(t) &= w(0)
  \end{align*}
  giving us
  \[
    \begin{pmatrix}
      e^{-3/2} & \frac{2}{3}e^{-3/2}-\frac{2}{3} & \frac{2}{3}e^{-3/2}-\frac{2}{3}\\
      0 & 1 & 0\\
      0 & 0 & 1
    \end{pmatrix}
  \]

  For the last segment, our parametrization is $(0,1-t,0)$, which gives us
  \begin{align*}
    u'(t) &= -u(t) - e^{1-t}v(t)\\
    v'(t) &= \frac{-3}{2}v(t) - w(t)\\
    w'(t) &= 0
  \end{align*}
  solving we get
  \begin{align*}
    u(t) &= Ae^{-t} + \frac{2e}{3}Be^{-5t/2} + \frac{2e}{3}Cte^{-t}\\
    v(t) &= Be^{-3t/2} - \frac{2}{3}C\\
    w(t) &= C
  \end{align*}
  giving us
  \[
    \begin{pmatrix}
      e^{-1} & \frac{2}{3}e^{-3/2}-\frac{2}{3} & \frac{2}{9} + \frac{4}{9}e^{-3/2}\\
      0 & e^{-3/2} & \frac{2}{3}e^{-3/2}-\frac{2}{3}\\
      0 & 0 & 1
    \end{pmatrix}
  \]

  The resulting matrix corresponding to this parallel loop is:
  \[
    \begin{pmatrix}
      1 & \frac{-2\left(\left(e^3+e^2-e+1\right)e^{1/2} + e^3-2e^2-e\right)}{3e^2} & \frac{-2\left(\left(2e^4+2e^3+5e^2+3e-1\right)e^{1/2} - 2e^4 - 8e^3 - e^2\right)}{9e^3}\\
      0 & 1 & 0\\
      0 & 0 & 1
    \end{pmatrix}
  \]

  In particular, this implies that $\mathfrak{hol}^\phi_{0}$ contains an element of the form
  \[
    \begin{pmatrix}
      0 & 1 & x\\
      0 & 0 & 0\\
      0 & 0 & 0
    \end{pmatrix}
  \]
  (note that the value of $x$ does not actually matter for our purposes).
\end{repex}


\bibliographystyle{amsalpha}
\bibliography{References}

\end{document}